\documentclass[reqno]{amsart}

\usepackage{hyperref}

\newtheorem{theorem}{Theorem}[section]
\newtheorem{lemma}[theorem]{Lemma}
\newtheorem{proposition}[theorem]{Proposition}
\newtheorem{corollary}[theorem]{Corollary}

\theoremstyle{definition}
\newtheorem{definition}[theorem]{Definition}

\newcommand{\la}{\lambda}
\newcommand{\In}{I_i^{n}}
\newcommand{\weak}{\rightharpoonup^*}
\newcommand{\prob}{\mathcal{P}(\overline{I})}
\newcommand{\Fi}{F_{\infty}}
\newcommand{\Fn}{F_n}

\def\XXint#1#2#3{{\setbox0=\hbox{$#1{#2#3}{\int}$} \vcenter{\vspace{-1pt}\hbox{$#2#3$}}\kern-.5\wd0}}
\def\Xint#1{\mathchoice {\XXint\displaystyle\textstyle{#1}}{\XXint\textstyle\scriptstyle{#1}}{\XXint\scriptstyle\scriptscriptstyle{#1}}{\XXint\scriptscriptstyle\scriptscriptstyle{#1}}\!\int}
\def\intmed{\Xint{-}}

\title{Spectral partitions for Sturm-Liouville problems}
\author{Paolo Tilli and Davide Zucco}
\address[Paolo Tilli]{Dipartimento di Scienze Matematiche, Politecnico di Torino, Italy}
\email{paolo.tilli@polito.it}
\address[Davide Zucco]{Istituto Nazionale di Alta Matematica, Unit\`a di Ricerca del Dipartimento di Scienze Matematiche, Politecnico di Torino, Italy}
\email{davide.zucco@polito.it}

\begin{document}

\begin{abstract}
We look for best partitions of the unit interval that minimize certain functionals defined in terms of the eigenvalues of Sturm-Liouville problems. Via $\Gamma$-convergence theory, we study the asymptotic distribution of the minimizers as the number of intervals of the partition tends to infinity. Then we discuss several examples that fit in our framework, such as the sum of (positive and negative) powers of the eigenvalues and an approximation of the trace of the heat Sturm-Liouville operator.
\end{abstract}

\maketitle

\section{Introduction}
We consider new minimization problems for the eigenvalues of Sturm-Liouville operators in bounded intervals. 
Throughout the paper, we assume we are given:
\begin{itemize}
\item[-] The unit interval $I:=(0,1)\subset \mathbb R$.
\item[-] A constant $1<\beta<+\infty$.
\item[-] A function $q\in L^\infty(\mathbb R)$ such that $0\leq q(x)\leq \beta$ for a.e. $x\in \mathbb R$.
\item[-] Two functions $p,w\!\in\! L^\infty(\mathbb R)$ such that $1/\beta\leq p(x)\leq \beta$ and $1/\beta \leq w(x)\leq\beta$. 
It is also convenient to introduce the function $s\in L^\infty(\mathbb R)$ defined by 
\begin{equation}\label{s}
s(x):=\sqrt{\frac{w(x)}{p(x)}},
\end{equation}
(note that $1/\beta\leq s(x)\leq \beta$).
\item[-] A function $\varphi\colon [0,+\infty)\to [0,+\infty]$ strictly convex, lower semicontinuous and with \emph{non-linear} growth, such that
$$\varphi^\infty:=\lim_{t\to+\infty} \frac{\varphi(t)}{t}=\sup_{t>0}  \frac{\varphi(t)}{t}\in \{0,+\infty\},$$
i.e., its recession factor is either null or infinite.
\end{itemize}
Under these assumptions, for any open bounded set $A\subset \mathbb R$ ($A\neq\emptyset$) we may define the \emph{first eigenvalue} of the Sturm-Liouville operator with coefficients $p$, $q$, $w$, via its variational characterization
\begin{equation}\label{eigenvalue}
\la(A):= \min_{\begin{subarray}{c}u\in H^1_0(A)\\ u\neq 0\end{subarray}} \frac{\int_{A} p(x)u'(x)^2dx+ \int_A q(x) u(x)^2 dx}{\int_{A}w(x)u(x)^2 dx}.
\end{equation}
The \emph{first eigenfuction} $u_A$ is a non-negative function (unique, up to a multiplicative constant, when $A$ is connected) which minimizes this quotient and, in particular, solves (in a weak sense) the Sturm-Liouville equation 
\begin{equation*}
-(p u_{A}')'+qu_{A}=\la({A})w u_{A}, \quad \text{in $A$}
\end{equation*}
with the boundary Dirichlet condition $u_{A}=0$ on $\partial A$ (see \cite{couhil} and also \cite{atkmin,zettl} for extended presentations on the Sturm-Liouville theory). As usual, when $A=\emptyset$ we set $\la(A):=+\infty$.

Moreover, for any positive integer $n$ we introduce the class
\[
\mathcal C_n:=\left\{\{I_j\}_{j=1}^n: \  \text{$I_j=(x_{j-1},x_{j})$ with $x_{j-1}\leq x_{j}$ for all $j$, $x_0:=0$ and $x_{n}:=1$}\right\} 
\]
of partitions of the unit interval $I$ made up of $n$ open intervals (notice that some intervals
may be empty, while the non-empty ones are necessarily disjoint). Then, to any interval $I_j$ of a partition in $\mathcal C_n$, we may associate the real number $\la(I_j)$, corresponding to the first eigenvalue of the set $I_j$ as defined in \eqref{eigenvalue} by choosing $A=I_j$, and build up minimization problems with functionals defined in terms of these eigenvalues.
More specifically, for a given natural number $n$ we study the minimization problem
\begin{equation}\label{problem}
\min \bigg\{\frac{1}{n}\sum_{1\leq j\leq n} \varphi\left(\frac{n}{\la(I_j)^{1/2}}\right) : \ \{I_j\}\in \mathcal C_n \bigg\},
\end{equation}
namely we are interested in the best partition of the unit interval, by means of $n$ intervals, so as to minimize a cost functional defined in terms of the eigenvalues of a Sturm-Liouville operator. Any partition $\{I_j\}\in \mathcal C_n$ can be induced by a set of $n-1$ (not necessarily distinct) points in $\overline I$ and one may equivalently regard \eqref{problem} with points as unknowns. Therefore, according to whether one wishes to optimize among partitions or points, problem \eqref{problem} is a matter of optimal \emph{partition} or of optimal \emph{location} (see
\cite{caflin,coteve,hehote} for some \emph{optimal partition problems} with cost functionals depending on the eigenvalues of the Laplacian in higher dimension and \cite{bojima2, busava, lmsz, suzdre, suzoka} for some \emph{optimal location problems} of non-spectral cost functionals). Problem \eqref{problem} can also be seen as a one-dimensional version of the problems introduced in \cite{henzuc, tilzuc, tilzuc2}, where the issue was how to best place a Dirichlet boundary condition in a two-dimensional membrane in order to optimize the first eigenvalue of an elliptic operator.

Some possible physical interpretations of problem \eqref{problem} are as follows.
In acoustics, $I$ represents a non-homogeneous string (of density $w$, Young modulus $p$ and subjected to the potential $q$) fixed at its endpoints with a frequency of vibration proportional to the first eigenvalue ${\la(I)^{1/2}}$. 
By adding $n-1$ extra points (nails) in the middle, where the string will be supplementarily fixed, the whole string will then vibrate according to the $n$ independent substrings $I_j$. 
Therefore, we are asking where best to locate the points so as to reinforce a string by optimizing suitable combinations of the frequencies of vibration of each substring (actually we consider reciprocal frequencies, namely {periods}). 
Moreover, in the framework of quantum mechanics, the Sturm-Liouville operator with $p=w=1$ reduces to the Schr\"odinger operator. The first eigenvalue $\la(I)$ corresponds to the \emph{ground state} of a quantum mechanical system is its lowest-energy state, namely it represents the lower energy level, in atomic units, of a quantum particle with a potential $q$ which becomes infinite outside $I$. Therefore, we wonder how to best trap quantum particles in $n$ subregions $I_j$'s to optimize their ground states. 

{The existence of an optimal partition for \eqref{problem} follows from the lower semicontinuity of $\varphi$ and the definition of $\mathcal C_n$. Moreover, by strict convexity of $\varphi$, every optimal partition is necessarily made up of $n$ \emph{distinct} open intervals, i.e. it has no empty intervals (but we do not expect uniqueness of the minimizer for arbitrary data $p,q,w$ and $\varphi$). It is however not clear if there is some monotonicity with respect to $n$. This also motivates the aim of the paper: analyze the asymptotic distribution of the sets $I_j$ inside the unit interval $I$ as $n\to\infty$.} As $n$ increases, any information concerning the density, i.e., number of intervals of $I_j$'s for unit length, is lost.
For optimal location problems (see for instance \cite{bojima2, busava}) a common strategy used to retrieve this information is to prove a $\Gamma$-convergence result in the space of probability measures, identifying each set of $n-1$ points with the sum of $n-1$ Dirac deltas supported on this set of points. In our setting a natural choice would be to associate a Dirac delta centered at  any interval of the partition. However, as we will see in the proof of the $\Gamma$-convergence result (see Section~\ref{sec.3}), it is more convenient to associate a probability measure that is concentrated on the whole interval $\overline I$. For this reason to any partition $\{I_j\}\in \mathcal C_n$ we associate the probability measure
\begin{equation}\label{meas}
 \mu_{\{I_j\}}=\frac{1}{n}\sum_{\begin{subarray}{c} 1\leq j\leq n\\ x_{j-1}<x_{j}\end{subarray}}\frac{1}{\mathcal L(I_j)}\chi_{I_j}(x) \mathcal L+\frac{1}{n}\sum_{\begin{subarray}{c} 1\leq j\leq n\\ x_{j-1}=x_{j}\end{subarray}}\delta_{x_j},
\end{equation}
where $\mathcal L$ is the Lebesgue measure and $\delta_x$ the Dirac delta at $x$. Namely, we define a measure which has constant density $1/\mathcal L(I_j)$ on every interval $I_j$ of the partition (when the length of the interval is zero this has to be meant as Dirac delta).
The normalization factor $n^{-1}$ in \eqref{meas} provides a probability measure. 
Moreover, rescaling by $n$ in \eqref{problem} serves to prevent loss of information in the limit; indeed the outer rescaling has the effect of averaging the sum of the $n$ terms, the inner one to balance intervals for which the eigenvalue grows as $n^2$ (this happens for instance when the partition is equidistribuited inside $I$).

The main result of the paper is then the following.

\begin{theorem}\label{theorem}
As $n\to\infty$ the functionals $\Fn\colon \prob\to  [0,+\infty]$ defined as
\begin{equation}\label{Fn}
\Fn(\mu)\!:=
\!\begin{cases}
\displaystyle \frac{1}{n}\!\sum_{1\leq j\leq n}\! \varphi\!\left(\!\frac{n}{\la(I_j^n)^{1/2}}\!\right) &\text{if $\mu=\mu_{\{I_j^n\}}$ as in \eqref{meas}, for $\{I_j^n\}\in\mathcal C_n$,}\\
\quad\; +\infty \quad &\text{otherwise},
\end{cases}
\end{equation}
$\Gamma$-converge, with respect to the weak* convergence in the space of probability measures $\mathcal P(\overline I)$, to the functional $F_\infty\colon\prob\to[0,+\infty]$ defined as
\begin{equation}\label{Fi}
\Fi(\mu):=\int_I \varphi\left(\frac{s(x)}{\pi f(x)}\right) f(x)\, dx+\varphi^\infty\mathcal L(\{f=0\})+ \varphi(0)\mu^s(\overline I),
\end{equation}
where $\mu=f\mathcal L+\mu^s$ is the decomposition of $\mu$ with respect to the Lebesgue measure ($f\mathcal L$ with $f\in L^1(I)$ is the absolutely continuous part of $\mu$ and $\mu^s$ its singular part) and the function $s(x)$ defined by \eqref{s}.
\end{theorem}

With an abuse of notation, the $\Gamma$-limit \eqref{Fi} can be shortly written as
\[
 \Fi(\mu)=\int_{\overline I} \varphi\left(\frac{s(x)}{\pi f(x)} \right)d\mu(x),
\]
where the integration w.r.t the measure $\mu$ has indeed to be intended as \eqref{Fi}.
Note that when $\varphi(0)=+\infty$ the $\Gamma$-limit is  finite only for absolutely continuous measures whereas when $\varphi$ is superlinear, i.e., $\varphi^\infty=+\infty$ the $\Gamma$-limit is finite only for measures with positive densities over all $I$.

Now, by classical results of $\Gamma$-convergence theory (see \cite{dalmaso}) combined with Jensen inequality we obtain the following information on the asymptotic behavior of the minimizing sequences of \eqref{problem}.

\begin{corollary}\label{corollary}
As $n \to \infty$, if $\{I_j^n\}$ is a minimizing sequence of problem \eqref{problem}, then  the probability measures $\mu_{\{I_j^n\}}$ converge in the weak* topology on $\prob$ to the probability measure $\mu_\infty$, absolutely continuous with respect to the Lebesgue measure, with density 
given by
\begin{equation}\label{minimum}
f_\infty(x)=\frac{s(x)}{\int_I s(t)dt},
\end{equation}
with $s(x)$ the function defined in \eqref{s}. In particular, for every open set $A\subset I$,
\begin{equation}\label{portion}
\lim_{n\to\infty} \sum_{\begin{subarray}{c} 1\leq j\leq n\\ x_{j-1}<x_{j}\end{subarray}}\frac{\mathcal L(I_j^n\cap A)}{n\mathcal L(I_j^n)}+\frac{\sharp\left(\{j: \text{$x_j^n\in A$ and $x_{j-1}^n=x_{j}^n$}\}\right)}{n}=\int_A f_\infty(x)dx, 
\end{equation}
where $I_j^n=(x_{j-1}^n,x_{j}^n)$ and
\[
 \lim_{n\to\infty} \frac{1}{n}\sum_{1\leq j\leq n}\varphi\left(\frac{n}{\la(I_j^n)^{1/2}}\right)=\varphi\left(\frac{1}{\pi}\int_Is(x)dx\right).
\]
\end{corollary}

By \eqref{portion} we deduce that the number of intervals of an optimal partition that intersect $A$, rescaled by $n$, converges to the density $f_\infty$.
This means that in order to reinforce a non-homogeneous string it is asymptotically convenient to concentrate the partition in those regions of the string at higher density $w$ and lower Young modulus $p$, with a density $f_\infty$  given by \eqref{minimum}, namely proportional to $\sqrt{w/p}$ (when the string is homogeneous, as expected, one finds in the limit the uniform distribution). Observe that the limiting measure $f_\infty dx$ is absolutely continuous with respect to the Lebesgue measure and that it neither depends on the potential $q$ nor on the convex function $\varphi$, reflecting the fact that in all cases minimizing sequences exhibit the same asymptotic behavior as $n\to\infty$.  However, the convex function $\varphi$ appears in limit of the cost functionals.

Some comments on the assumptions on $\varphi$ are in order. For Theorem~\ref{theorem} strict convexity can be relaxed to convexity. The strict convexity is just needed for the uniqueness of the minimizer of \eqref{Fi}, that is used in Corollary~\ref{corollary} (otherwise a similar statement holds up to subsequences). Moreover, the positivity of $\varphi$ can  easily be replaced by a lower bound on $\varphi$, that is by requiring the minimum of $\varphi$ on $[0,+\infty)$ to be possibly negative but finite. Moreover, the case where $\varphi$ has linear growth, i.e., $0<\varphi^\infty<+\infty$, deserves deeper inspections (Lemma~\ref{lem.linear} is a first attempt in this direction). Indeed when the function $\varphi$ is linear the analysis becomes more complicated since the $\Gamma$-limit should also depend on some \emph{convexity} of the coefficients $p$ and $w$ and on the potential $q$. At least when $p$ and $w$ are sufficiently regular ($\mathcal C^1(\mathbb R)$ would be enough) and $\{f=0\}$ is an open interval we conjecture the following dichotomy to hold. The second term in the $\Gamma$-limit \eqref{Fi} should be replaced with:
\begin{itemize}
\item[-] ${\varphi^\infty}/{\la(\{f=0\})^{1/2}}$ when $w h''+qh\geq 0$ over $\{f=0\}$;
\item[-] $({\varphi^\infty}/{\pi})\int_{\{f=0\}} s(x)dx$ when $w h''+qh< 0$ over $\{f=0\}$;
\end{itemize}
where $h=\sqrt[4]{pw}$. Notice that, in the case of constant coefficients and $q=0$ these two terms coincide (see also the comment at the end of Example 1. in Section~\ref{sec.4}). 
In general, when the coefficients are only measurable and $f$ vanishes on a generic measurable set, 
it is not clear what the explicit representation of this term should be. It is worth noting that this dichotomy is interestingly connected to the following inequality {\`a la} Brascamp-Lieb \cite{bralie}:
\[
\frac{1}{\la((0,x))^{1/2}}+\frac{1}{\la((x,1))^{1/2}}\geq \frac{1}{\la(I)^{1/2}}\quad  \text{for every $x\in I$}.
\]
We therefore want to throw down a gauntlet and ask what are sharp conditions on the coefficients $p,q,w$ for the validity of this estimate (in the case of constant coefficients $p,q,w$ it is always true when $q\geq 0$, but becomes false as soon as $q$ is allowed to be negative).

We finally notice that Theorem~\ref{theorem} can easily be extended to a finite union of intervals. For the case of an unbounded interval (for some particular choices of data $p,q,w$) some extra difficulties emerge but the problem can be of interest as well (this is somehow related to the case of unbounded non-negative data, which corresponds to letting $\beta=+\infty$).
\smallskip

The plan of the paper is the following. In Section~\ref{sec.2} we prove some preliminary results on the eigenvalues of Sturm-Liouville operators. In Section~\ref{sec.3} we prove $\Gamma$-convergence results, in particular Theorem~\ref{theorem} and Corollary~\ref{corollary}. In section~\ref{sec.4} we show several concrete examples for which Theorem~\ref{theorem} and Corollary~\ref{corollary} apply.

\subsection*{Notation}
Here and henceforth $\sharp$ and $\mathcal L$ denote respectively the counting measure and the Lebesgue measure. Integration with respect to the Lebeasgue measure will be simply denoted by $dx$. The usual abbreviations w.r.t. and $\mu$-a.e. stand for ``with respect to'' and ``almost everywhere with respect to the measure $\mu$''. If no measure is specified it has to be intended w.r.t. the Lebesgue measure. We adopt the following conventions $c(+\infty)=+\infty$ and $c/(+\infty)=0$ for all constants $c>0$.
We denote by $\intmed_A f :=\frac{1}{\mathcal L(A)}\int_A f(x) dx$ the mean value of a function $f\in L^1(A)$ on some open set $A\subset I$.  When $J\subset \mathbb R$ is a non-empty open interval and $x\in \mathbb R$ with $\lim_{J\downarrow x}$ we mean the limit as $J$ shrink toward $x$. With some abuse of notation if $\eta\in \prob$ then by $\eta (J)$ we always mean $\eta(J\cap \overline I)$.

\section{Preliminaries on the eigenvalues of Sturm-Liouville problems}\label{sec.2}

We introduce some preliminary results that will be used for the proof of the main theorem. 
In this section $J \subset \mathbb R$ will denote a non-empty bounded open interval. Then when $p$, $q$, and $w$ are constants the eigenvalue defined in \eqref{eigenvalue} with $A=J$ admits the following explicit expression:
\begin{equation}\label{explicit}
\la(J)=\frac{\pi^2}{s^2\mathcal L(J)^2}+q,
\end{equation}
where we recall that, by definition \eqref{s}, $s^2=w/p$. We point out that the Sturm-Liouville problem reduces to the pure Laplace problem when the coefficients are $p=w=1$ (thus $s=1$) and $q=0$.
In general, when the coefficients are not necessarily constants, the eigenvalue \eqref{eigenvalue} cannot be explicitly computed in terms of geometrical quantities. However the eigenvalue of an interval is related to the reciprocal of its length squared. More precisely, the assumptions on the coefficients $p$, $q$ and $w$ provide the following global bounds:
\begin{equation}\label{global}
\frac{1}{\beta^2} \frac{\pi^2}{\mathcal L(J)^2}\leq {\la(J)} \leq \beta^2\frac{\pi^2}{\mathcal L(J)^2}+\beta^2.
\end{equation}
Locally we can be more precise: we prove that its local behavior is the same as when the coefficients are constant.

\begin{lemma}\label{lem.local}
Let $x_0$ be a Lebesgue point for $p$ and $w$. Then
\[
\lim_{J\downarrow x_0} {\la(J)\mathcal L(J)^2}=\frac{\pi^2}{s(x_0)^2}.
\]
\end{lemma}
\begin{proof}
We first prove upper and lower bounds by means of classical change of variables (see, e.g., \cite[p. 292]{couhil}). For the lower bound, if $J=(a,b)$ the function $\phi_1: J\to J$, defined for every $x\in J$ as
\[
\phi_1(x):=\frac{1}{\intmed_Jw}\int_a^x w(t)\,dt+a,
\]
allows to change variable $y=\phi_1(x)$ in \eqref{eigenvalue} and obtain
\[
\lambda(J)=\min_{\begin{subarray}{c}u\in H^1_0(J)\\ u\neq 0\end{subarray}} \frac{(\intmed_J w)^{-2}\int_J p_1(y) w_1(y)u'(y)^2 dy+{\int_J (q_1(y)/w_1(y))u(y)^2 dy}}{\int_J u(y)^2 dy},
\]
where we set
\[
p_1(y):=p(\phi_1^{-1}(y)), \quad q_1(y):=q(\phi_1^{-1}(y)),\quad w_1(y):=w(\phi_1^{-1}(y)).
\]
Then, by testing the minimum with the first eigenfunction $u_1$ of the Laplacian in $J$, normalized as $\int_J u_1(y)^2dy=1$ (e.g., when $J=(0,b)$ take $v(x)=u_1(x):=\sqrt{2/b}\sin(\pi x/b)$), together with the trivial estimates $p_1(y) w_1(y)\leq\intmed_J(p_1w_1)+|p_1(y) w_1(y)-\intmed_J(p_1w_1)|$ and $q_1(y)/ w_1(y)\leq \beta^2$ for a.e. $y\in J$ allows to obtain
\[
{\la(J)}\leq \frac{1}{(\intmed_J w)^2}\Big(\intmed_J(p_1w_1)\int_J u_1'(y)^2dy+\int_{J} \big|p_1(y)w_1(y)-\intmed_J(p_1w_1)\big|u_1'(y)^2dy\Big)+\beta^2.
\]
By \eqref{explicit} we have $\int_J u_1'(y)^2dy= \pi^2/\mathcal L(J)^2$, and since $\sup_{J} (u_1')^2={2\pi^2}/{\mathcal L(J)^3}$ we find
\[
\la(J)\leq \frac{\pi^2}{\mathcal L(J)^2} \frac{1}{(\intmed_J w)^2}\bigg(\intmed_J(p_1w_1)+2\intmed_J\Big|p_1w_1-\intmed_J(p_1w_1)\Big|\bigg)+\beta^2.
\]
Exploiting the change of variables given by $\phi_1$ one obtains the following
\begin{equation}\label{upper}
\la(J)\leq \frac{\pi^2}{\mathcal L(J)^2}\frac{1}{(\intmed_J w)^{3}}\bigg(\intmed_J(pw^2)+2\intmed_J\Big|pw^2-\frac{\intmed_J(pw^2)}{\intmed_J w}w\Big|\bigg)+\beta^2.
\end{equation}

For the lower bound, we use instead the function $\phi_2: J\to J$, defined for every $x\in J$ as
\[
\phi_2(x):=\frac{1}{\intmed_J 1/p}\int_a^x 1/p(t)\,dt+a,
\]
to change variables $z=\phi_2(x)$ in \eqref{eigenvalue} and obtain
\[
\lambda(J)=\min_{\begin{subarray}{c}u\in H^1_0(J)\\ u\neq 0\end{subarray}} \frac{(\intmed_J 1/p)^{-2}\int_J u'(z)^2 dz+{\int_J p_2(z)q_2(z)u(z)^2 dz}}{\int_J p_2(z)w_2(z)u(z)^2 dz},
\]
where here we set
\[
p_2(z):=p(\phi_2^{-1}(z)), \quad q_2(z):=q(\phi_2^{-1}(z)),\quad w_2(z):=w(\phi_2^{-1}(z)).
\]
Now, if $u_J$ is the minimizer of the above quotient with $\int_J (u_J)^2dz=1$, the trivial estimates  $p_2(z)q_2(z)\geq0$ and $p_2(z) w_2(z)\leq\intmed_J(p_2w_2)+|p_2(z) w_2(z)-\intmed_J(p_2w_2)|$ for a.e. $z\in J$ yield
\[
\la(J)\geq  \frac{\int_{J} u_J'(z)^2dz}{(\intmed_J 1/p)^{-2}\big(\intmed_J(p_2w_2)+\int_{J} |p_2(z)w_2(z)-\intmed_J(p_2w_2)|u_J(z)^2dz\big)}.
\]
By using first at denominator the one-dimensional Poincar\'e inequality $\sup_J (u_J)^2\leq \mathcal L(J)\int_J (u_J')^2dz$ and then for the numerator $\int_J (u_J')^2dz\geq \pi^2/\mathcal L(J)^2 \int_J (u_J)^2dz$ (recall the eigenvalue of the Laplace problem \eqref{explicit} in the case $p=w=1$ and $q=0$) we obtain
\begin{equation*}
\la(J)\geq \frac{\pi^2}{(\intmed_J 1/p)^{2}\big(\intmed_J(p_2w_2)+\pi^2\intmed_{J} |p_2w_2-\intmed_J(p_2w_2)|\big)\mathcal L(J)^2},
\end{equation*}
and now, exploiting the change of variables given by $\phi_2$ we obtain
\begin{equation}\label{lower}
\la(J)\geq \frac{\pi^2}{(\intmed_J 1/p)\big(\intmed_J(w)+\pi^2\intmed_{J} \big|w-\frac{\intmed_J(w)}{\intmed_J 1/p}\frac{1}{p}\big|\big)\mathcal L(J)^2}.
\end{equation}
The result stated in the lemma then follows by combining \eqref{upper} with \eqref{lower}, thanks to Radon-Nikodym theorem (see \cite{rudin}).
\end{proof}

We now prove an asymptotic result for the sum of reciprocal eigenvalues. 
For this we need the following definition.

\begin{definition}[Accumulation point]\label{def.acc}
Let $\{I_j^n\}\in \mathcal C_n$ for all $n\in\mathbb N$. We say that a point $x\in I$ is an \emph{accumulation point} of the partitions $\{I_j^n\}$ if for every $n\in\mathbb N$ there exists an index $j_n$ with $x\in \overline{I}_{j_n}^{\,n}$ and $I_{j_n}^n\in\{I_j^n\}$ such that
$$\lim_{n\to\infty} \mathcal L(I_{j_n}^n)=0.$$
We denote by $I^a$ the set made up of all accumulation points of the partitions $\{I_j^n\}$ and by $I^{b}=I\setminus I^a$ its complement.
\end{definition}

For instance, if $I_j^n=((j-1)/(2n-2), j/(2n-2))$ for all $j=1,\dots (n-1)$ while $I_n^n=(1/2,1)$ then $I^a=(0,1/2]$ and $I_b=(1/2,1)$.
In other words a point belongs to $I^b$ if there are no sequences of intervals shrinking toward it. This implies that $I^b$ is made up of union of open intervals and so it is an open set; then both $I^a$ and $I^b$ are measurable sets. 

\begin{lemma}\label{lem.linear}
Let $\{I_j^n\}\in \mathcal C_n$
and $J\subset \mathbb R$ be a non-empty bounded open interval. 
Then,  up to subsequences, 
\[
\lim_{n\to\infty} \sum_{1\leq j\leq n\atop x_{j-1}<x_{j}} \frac{\mathcal L(I_j^n\cap J)}{\lambda(I_j^n)^{1/2}\mathcal L(I_j^n)}= \frac{1}{\pi}\int_{I^a\cap J} s(x)\, dx+c_J \mathcal L(I^b\cap J),
\]
where $c_J$ is a constant possibly depending on $J$ such that $1/(\beta\sqrt{\pi^2+1})\leq c_J\leq \beta/\pi$.
\end{lemma}
\begin{proof}
To compute the limit we split the sum w.r.t. those indeces inside and outside the set of accumulation points, i.e., we write
\begin{equation}\label{proof.linear}
\sum_{1\leq j\leq n\atop x_{j-1}<x_{j}} \frac{\mathcal L(I_j^n\cap J)}{\lambda(I_j^n)^{1/2}\mathcal L(I_j^n)}= \sum_{1\leq j\leq n\atop x_{j-1}<x_{j}} \frac{\mathcal L(I_j^n\cap I^a\cap J)}{\lambda(I_j^n)^{1/2}\mathcal L(I_j^n)}+ \sum_{1\leq j\leq n\atop x_{j-1}<x_{j}}\frac{\mathcal L(I_j^n\cap I^b\cap J)}{\lambda(I_j^n)^{1/2}\mathcal L(I_j^n)}.
\end{equation}
Since a.e. point in $\mathbb R$ is a Lebesgue point for $p$ and $w$, by Definition~\ref{def.acc} and Lemma~\ref{lem.local} with the Lebesgue density theorem
\[
\lim_{n\to\infty}\sum_{1\leq j\leq n\atop x_{j-1}<x_{j}}\frac{\mathcal L(I_j^n\cap I^a\cap J)}{{\la(I_j^n)^{1/2}}\mathcal L(I_j^n)^2}\,\chi_{I_j^n}(x)= \frac{s(x)}{\pi}\quad \text{a.e. $x\in I^a\cap J$},
\]
Notice that the set of points of $I_a$ for which the Lebesgue density theorem does not apply is at most countable 
(indeed every such bad point is surrounded by two open subintervals of $I_b$, and by definition $I_b$ is at most countable).
Moreover, by the lower bound in \eqref{global},
\[
\lim_{n\to\infty}\sum_{1\leq j\leq n}\frac{\mathcal L(I_j^n\cap I^a\cap J)}{{\la(I_j^n)^{1/2}}\mathcal L(I_j^n)^2}\,\chi_{I_j^n}(x)\leq \frac{\beta}{\pi}\quad \text{a.e. $x\in I^a\cap J$}.
\]
Then the Lebesgue dominated convergence theorem implies that the limit of the first term in the right-hand side of \eqref{proof.linear} exists and
\[
\lim_{n\to\infty}\sum_{1\leq j\leq n}\frac{\mathcal L(I_j^n\cap I^a\cap J)}{{\la(I_j^n)^{1/2}}\mathcal L(I_j^n)}= \frac{1}{\pi} \int_{I^a\cap J}s(x)\,dx.
\]

For the second term in the right-hand side of \eqref{proof.linear} we use \eqref{global} with $J=I_j^n$ and the trivial estimate $\mathcal L(I_j^n)\leq 1$ for all $j$ to bound it as follows: 
\[
\frac{1}{\beta\sqrt{\pi^2+1}}\mathcal L(I^b\cap J) \leq\sum_{1\leq j\leq n}\frac{\mathcal L(I_j^n\cap I^b\cap J)}{\lambda(I_j^n)^{1/2}\mathcal L(I_j^n)}\leq  \frac{\beta}{\pi} \mathcal L(I^b\cap J).
\]
Therefore, up to subsequences, also this term converges as $n\to+\infty$ and what claimed in the lemma holds.
\end{proof}

The previous lemma can be thought of as a first result for the asymptotics in the linear case where $\varphi(x)=\varphi^\infty x$ with $0<\varphi^\infty<+\infty$. In this case the difficulty relies on the computation of the second term, which depends on the coefficients $p,q$ and $w$.

We conclude these preliminaries with some local estimates for the set of all accumulation points (and consequently for its complement).

\begin{lemma}\label{lem.point}
Let $\{I_j^n\}\in \mathcal C_n$ such that the measures $\mu_{\{I_j^n\}}$ as defined in \eqref{meas} weak* converge to a measure $\mu\in \prob$. If $x_0\in I$ is a Lebesgue point for the Radon-Nikodym derivative $f\in L^1(I)$ of $\mu$ w.r.t. the Lebesgue measure with $f(x_0)>0$ then $$\lim_{J\downarrow x_0}\frac{\mathcal L(I^a\cap J)}{\mathcal L(J)}=1 \quad \text{and}\quad  \lim_{J\downarrow x_0}\frac{\mathcal L(I^b\cap J)}{\mathcal L(J)}= 0.$$
\end{lemma}
\begin{proof}
Fix $0<\epsilon<f(x_0)$. Then for every $J$ sufficiently small around $x_0$ we have
$$\mathcal L(\{x\in J : |f(x)-f(x_0)|>\epsilon\})<\epsilon \mathcal L(J)$$
and then, passing to the complementary set,
$$\mathcal L(\{x\in J : |f(x)-f(x_0)|\leq \epsilon\})\geq (1-\epsilon) \mathcal L(J).$$ 
Since the inclusion $\{x\in J : |f(x)-f(x_0)|\leq\epsilon\}\subset I^a\cap J$ holds up to negligible sets then $(1-\epsilon)\leq \mathcal L(I^a\cap J)/\mathcal L(J)$. 
Letting $J\downarrow x_0$ we obtain the claim, thanks to the arbitrariness of $\epsilon$.
\end{proof}

\section{$\Gamma$-convergence}\label{sec.3}

We start by proving the $\Gamma$-liminf and the $\Gamma$-limsup inequalities.

\begin{proposition}[$\Gamma$-liminf inequality]\label{prop.liminf}
For every probability measure $\mu\in\prob$ and every sequence $\{\mu_n\}\subset\prob$ such that $\mu_n\weak\mu$ it holds 
\begin{equation}\label{liminf}
  \liminf_{n\to\infty} \Fn(\mu_n)\geq \Fi(\mu).
\end{equation}
\end{proposition}
\begin{proof}
For a given $\mu\in\prob$ let $\mu=f\mathcal L+\mu^s$ be the Lebesgue decomposition of $\mu$ w.r.t. the Lebesgue measure, where $f\in L^1(I)$ with $f\geq0$ and $\mu^s$ its singular part. Given $\{\mu_n\}\subset\prob$ such that $\mu_n\weak\mu$, without loss of generality, we may consider an arbitrary subsequence $\{\mu_n\}$ (not relabelled) for which, 
by $\eqref{Fn}$,
\begin{equation*}
\liminf_{n\to\infty}\Fn(\mu_n)= \lim_{n\to\infty}\frac{1}{n}\sum_{1\leq j\leq n}\varphi\left(\frac{n}{\la(I_j^n)^{1/2}}\right)
\end{equation*}
exists, is finite and such that each $\mu_n$ has the form \eqref{meas} for some partition $\{I_j^n\}\in\mathcal C_n$ (since the inequality \eqref{liminf} is trivial along sequences $\{\mu_n\}_n$ for which the energy is always $+\infty$).
The Radon measures $\nu _n$ defined on $\overline I$ as 
\[
\nu_n=\frac{1}{n}\sum_{1\leq j\leq n\atop x_{j-1}<x_{j}}\varphi\left(\frac{n}{\la(I_j^n)^{1/2}}\right)\frac{1}{\mathcal L(I_j^n)}\chi_{I_j^n}(x) \mathcal L+\frac{1}{n}\sum_{1\leq j\leq n\atop x_{j-1}=x_{j}}\varphi(0)\delta_{x_j},
\]
satisfy $\sup_{n} \nu_n(\overline I)=\sup_{n}F_n(\mu_n)<+\infty$. This implies (see \cite[Section~1.9]{evagar}) the existence of a positive Radon measure $\nu$ defined on $\overline I$ and of a subsequence $\{\nu_n\}$ (not relabelled) for which $\nu_n\weak \nu$; in particular, 
since the measures are concentrated in $\overline I$, 
\begin{equation}\label{final}
\liminf_{n\to\infty}\Fn(\mu_n)=\lim_{n\to\infty}\nu_n(\overline I)= \nu(\overline I)\geq \int_{\overline I} g_\epsilon(x) d\mu_\epsilon, 
\end{equation}
where $g_\epsilon$ is the density of $\nu$ w.r.t the measure $\mu_\epsilon:=(f+\epsilon)\mathcal L+\mu^s$ for some $\epsilon>0$, see \cite[Section 1.6]{evagar} (adding $\epsilon$ saves possible regions where $f$ vanishes). It is then enough to estimate from below the density $g_\epsilon$ $\mu_\epsilon$-almost everywhere. To do this, we fix an open interval $J\subset \mathbb R$ such that its boundary $\partial J$ is not charged by $\mu$, namely $\mu(\overline J)=\mu(J)$. This is equivalent to require that $\partial J$ is not contained in the \emph{countable} set where the pure point part of $\mu$ is concentrated (note that this does not inquire constraints on the length of $J$). By $\mu_n\weak \mu$ 
\begin{equation}\label{equality}
\lim_{n\to\infty}\mu_n(J)=\mu(J)=\mu(\overline J),
\end{equation}
and by $\nu_n\weak\nu$ we have
\begin{equation}\label{estnu}
\limsup_{n\to\infty}\sum_{1\leq j\leq n\atop x_{j-1}<x_{j}}\frac{\mathcal L(I_j^n\cap J)}{n\mathcal L(I_j^n)}\varphi\bigg(\frac{n}{{\la(I_j^n)^{1/2}}}\bigg)+\frac{k}{n}\varphi\left(0\right)= \limsup_{n\to\infty}\nu_n(J)\leq \nu (\overline{J}),
\end{equation}
where $k$ is the number of intervals $I_j^n$'s in $J$ such that $x_{j-1}=x_{j}$. Since by \eqref{meas}
\[
\mu_n(J)=\sum_{1\leq j\leq n\atop x_{j-1}<x_{j}} \frac{\mathcal L(I_j^n\cap J)}{n\mathcal L(I_j^n)}+\frac{k}{n}>0,
\] 
Jensen inequality yields
\begin{equation}\label{mnmn}
\sum_{1\leq j\leq n\atop x_{j-1}<x_{j}}\!\!\!\!\!\!\frac{\mathcal L(I_j^n\cap J)}{n\mathcal L(I_j^n)}\varphi\Big(\frac{n}{{\la(I_j^n)^{1/2}}}\!\Big)\!+\!\frac{k}{n}\varphi\left(0\right)
\!\geq \!\mu_n(J)\varphi\Big(\frac{1}{\mu_n(J)}\!\!\!\!\!\sum_{1\leq j\leq n\atop x_{j-1}<x_{j}}\!\!\!\!\frac{\mathcal L(I_j^n\cap J)}{{\la(I_j^n)^{1/2}\mathcal L(I_j^n)}}\!\Big).
\end{equation}
Up to subsequences, we can assume the limits as $n\to\infty$ of \eqref{mnmn} and Lemma~\ref{lem.linear} to exist. Therefore, by letting $n\to\infty$ in \eqref{mnmn}, and by combining \eqref{estnu}, \eqref{equality}, and Lemma~\ref{lem.linear} with the continuity of $\varphi$ on $(0,+\infty)$, we obtain 
\[
\nu(\overline J)\geq \big(\mu(J)+\delta \mathcal L(J)\big)\varphi\left(\frac{1}{\mu(J)+\delta\mathcal L(J)}\frac{1}{\pi}\int_{I^a\cap J} s(x)dx+c_J\mathcal L(I^b\cap J)\right)-\epsilon,
\]
where $\delta>0$ is sufficiently small and serves to avoid vanishing quantities. 
Now, dividing by $\mu_\epsilon(\overline J)$ (which is always positive by definition) and letting $J$ shrink towards a $\mu_\epsilon$-Lebesgue point $x_0\in \overline I$ for $g_\epsilon$, from Radon-Nikodym theorem we can face three different situations:

\begin{itemize}
\item[1.] If $x_0$ is also a Lebesgue point for $p, w$ and $f$ with $f(x_0)>0$ then by Lemma~\ref{lem.point}, recalling the uniform bounds on the constant $c_J$ in Lemma~\ref{lem.linear},
\begin{equation*}
g_\epsilon(x_0)\geq \frac{f(x_0)+\delta}{f(x_0)+\epsilon}\,\varphi\left(\frac{s(x_0)}{\pi (f(x_0)+\delta)}\right)-\epsilon, 
\end{equation*}
and letting $\delta\to 0$ 
\begin{equation}\label{uno}
g_\epsilon(x_0)\geq \frac{f(x_0)}{f(x_0)+\epsilon}\,\varphi\left(\frac{s(x_0)}{\pi f(x_0)}\right)-\epsilon. 
\end{equation}

\item[2.] If instead $x_0$ is also a Lebesgue point for $p, w$ and $f$ with $f(x_0)=0$, up to subsequences as $J\downarrow x_0$, the right-hand side in Lemma~\ref{lem.linear} rescaled by $\mathcal L(J)$ converges to some constant $c>0$ such that $1/(\beta\sqrt{\pi^2+1})\leq c\leq \beta/\pi$, thanks to inequality \eqref{global} and Lemma~\ref{lem.local}. Then
\[
g_\epsilon(x_0)\geq \frac{\delta}{\epsilon}\,\varphi\left(\frac{c}{\delta}\right)-\epsilon,
\]
and letting $\delta\to 0$, from the fact that $\varphi^\infty$ is either $0$ or $+\infty$, we now obtain
\begin{equation}\label{due}
g_\epsilon(x_0)\geq \frac{1}{\epsilon}\varphi^\infty-\epsilon.
\end{equation}

\item[3.] If at $x_0$ the measure $\mu$ has Radon-Nikodym derivative $f(x_0)=+\infty$ then 
\begin{equation}\label{tre}
g_\epsilon(x_0)\geq \varphi(0)-\epsilon.
\end{equation}
\end{itemize}

Therefore, since a.e. $x_0\in I$ is a Lebesgue point for $p,w$ and $f$, while at $\mu^s$-a.e. point $x_0\in \overline I$ the Radon-Nikodym derivative $f$ is not finite (see \cite[Theorems 7.10 and 7.15]{rudin}), we can use \eqref{uno}, \eqref{due} and \eqref{tre} inside \eqref{final} to obtain
\[
\liminf_{n\to\infty}\Fn(\mu_n)\geq\int_I \varphi\left(\frac{s(x)}{\pi f(x)}\right) f(x)\, dx+\varphi^\infty\mathcal L(\{f=0\})+ \varphi(0)\mu^s(\overline I)-\epsilon\mu_\epsilon(\overline I),
\]
that is \eqref{liminf}, thanks to the bound $\mu_\epsilon(\overline I)\leq 1+\epsilon$ and the arbitrariness of $\epsilon$. 
\end{proof}

For the $\Gamma$-limsup inequality, we need to introduce  a special class of measures.
\begin{definition}[Piecewise constant measure]\label{def.step}
We say that a probability measure $\mu\in\prob$ is \emph{piecewise constant} if it is absolutely continuous w.r.t. the Lebesgue measure, with a density $f(x)\geq0$ constant on every interval $J_i:=((i-1)/m,i/m)$, with $i=1,\dots,m$, for some $m\in\mathbb N$. 
In formulae,
\begin{equation}\label{stepmeas}
\mu=f\mathcal L, \qquad f(x)= \sum_{1\leq i\leq m} \alpha_i \chi_{J_i}(x), \qquad \mathcal L(J_i)=1/m
\end{equation}
where the constants $\alpha_i\geq 0$ satisfy (since $\mu(I)=1$) the normalization condition 
\begin{equation*}
\sum_{1\leq i\leq m} \alpha_i= m.
\end{equation*}
Moreover, we also denote by $M_0$ the set of indeces corresponding to those intervals with $\alpha_i=0$ and $m_0=\sharp(M_0)$ its cardinality.
\end{definition}

Then to any piecewise constant measure we associate the following sequence.

\begin{definition}[The recovery sequence]\label{def.rec}
Consider a piecewise constant measure $\mu$, with the same notation as in Definition~\ref{def.step}.
We say that a sequence of probability measures $\{\mu_n\}$ in $\prob$ is a \emph{recovery sequence for $\mu$}, if for every $n\geq m$ (recall that $m$ is the characteristic parametric of $\mu$ defined in \eqref{stepmeas}) the following conditions are satisfied:
\begin{itemize}
\item[(i)] $\mu_n$ has the form  \eqref{meas} for some partition $\{I_j^n\}\in \mathcal C_n$ with $x_{j-1}< x_j$ for all $j=1,\dots n$ and
\[
\bigcup_{1\leq j\leq n}\partial I_j^n\supseteq \bigcup_{1\leq i\leq m}\partial J_i;
\]
\item[(ii)] when $i\in M_0$ there is only one interval of the partition $\{I_j^n\}$ in $J_i$ of \eqref{stepmeas} with $I_j^n=J_i$ (note that there are $m_0$ of such intervals of the partition), while when $i\notin M_0$  all the intervals of the partition $\{I_j^n\}$ in $J_i$ have same length and  the number $k_i$ of them is
\begin{equation}\label{kin}
k_i=k_i(n):= \left\lfloor \frac{\alpha_i}{m}(n-m_0)\right\rfloor + \gamma_i^{n},
\end{equation}
where $\gamma_i^{n}$ is a corrector factor, that is $0$ if the quantity $\alpha_in/ m$ is an integer else $0$ or $1$ accordingly to guarantee the condition
\begin{equation*}
\sum_{i\notin M_0} k_i=n-m_0.
\end{equation*}
\end{itemize}
\end{definition}

\begin{proposition}[$\Gamma$-limsup inequality]\label{proplimsup}
For every measure $\mu\in\prob$ there exists a sequence $\{\mu_n\}$ in $\prob$ such that $\mu_n\weak\mu$ and
\begin{equation}\label{limsup}
 \limsup_{n\to\infty} \Fn(\mu_n)\leq F_\infty(\mu).
\end{equation}
\end{proposition}
\begin{proof}
We first prove the result when $\mu$ is piecewise constant according to Definition~\ref{def.step} and consider a recovery sequence $\{\mu_n\}$ for $\mu$ as  in Definition~\ref{def.rec}. For every $n\geq m$ the measure $\mu_n$ has the form \eqref{meas} for some partition $\{I_j^n\}\in \mathcal C_n$ which covers all the boundaries of the intervals $J_i$'s. For every $i\in M_0$ we have $\mu_n(J_i)=\frac{1}{n}\sim 0=\mu(J_i)$ as $n\to\infty$ while if $i\notin M_0$ by \eqref{kin} with \eqref{stepmeas} yields
\begin{equation*}
\mu_n({J}_i)= \frac{k_i}{n} \sim \frac{\alpha_i}{m}=\mu(J_i)\quad  \text{as $n\to\infty$},
\end{equation*}
and this guarantees that the probability measures $\mu_{n}\weak \mu$ as $n\to\infty$. Moreover, by \eqref{Fn},
\begin{equation}\label{gammasupw}
\limsup_{n\to\infty} F_n(\mu_n)=
\limsup_{n\to\infty} \frac{1}{n}\sum_{i\notin M_0}\!\sum_{1\leq j\leq k_i}\! \varphi\Big(\frac{n}{\la(I_{i_j}^n)^{1/2}}\Big)\!+\frac{1}{n}\sum_{i\in M_0}\!\varphi\Big(\frac{n}{\la(J_i)^{1/2}}\Big).
\end{equation} 
Note that, by {\eqref{stepmeas}} and \eqref{kin} again, for every $i\notin M_0$ and $j=1,\dots, k_i$
\begin{equation}\label{construction}
\mathcal L(I_{i_j}^n)=\frac{1}{mk_i}\sim \frac{1}{\alpha_i(n-m_0)} \quad \text{as $n\to\infty$}.
\end{equation}
To estimate the first term in the right-hand side of \eqref{gammasupw}, fix $i\notin M_0$  and consider the functions $g_n$ defined for every $x\in I$ as
\[
g_n(x):=\frac{1}{n}\sum_{1\leq j\leq k_i} \varphi\left(\frac{n}{\la(I_{i_j}^n)^{1/2}}\right)\frac{1}{\mathcal L(I_{i_j}^n)}\chi_{I_{i_j}^n}(x).
\]
By combining Lemma~\ref{lem.local} and \eqref{construction} with the Radon-Nikodym theorem yields
\[
\lim_{n\to\infty} g_n(x)=\alpha_i \varphi\left (\frac{s(x)}{\pi \alpha_i}\right),\quad \text{a.e. $x\in J_i$,}
\]
and, by the Lebesgue dominated convergence theorem (notice that by \eqref{s} the function $\varphi(s(x)/\pi\alpha_i)$ is of course in $L^1(J_i)$) we obtain
\begin{equation}\label{11}
\lim_{n\to\infty} \int_{J_i} g_n(x)dx=\lim_{n\to\infty}\frac{1}{n}\sum_{1\leq j\leq k_i} \varphi\left(\frac{n}{\la(I_{i_j}^n)^{1/2}}\right) =\int_{J_i}\varphi\left (\frac{s(x)}{\pi \alpha_i}\right) \alpha_i\, dx.
\end{equation}
Now to estimate the latter term in \eqref{gammasupw}, fix $i\in M_0$ and let $n\to\infty$. Since $\varphi^\infty$ is either $0$ or $+\infty$ we have
\begin{equation}\label{22}
\frac{1}{n}\sum_{i\in M_0}\!\varphi\bigg(\frac{n}{\la(J_i)^{1/2}}\bigg)= \left(\frac{m_0}{m}\right)\varphi^\infty.
\end{equation}
Therefore, summing $i$ from $1$ to $m$ \eqref{limsup} follows by combining \eqref{11} with \eqref{22}, thanks to the fact that $\mu$ is a measure with a density $f(x)$ piecewise constant according to \eqref{stepmeas}.

Now, the passage to general $\mu$ is standard. By classical results of $\Gamma$-convergence theory it is enough to prove the \emph{density in energy} of the measures with piecewise constant densities in the space of probability measures, namely that for every $\mu\in\prob$ there exists a sequence $\{\mu_n\}$ such that $\mu_n$ is as in Definition~\ref{def.step} for some $\{I_j^n\}$, $\mu_n\weak\mu$ and 
\begin{equation}\label{claim1}
\limsup_{n\to\infty} \Fi(\mu_n)\leq \Fi(\mu).
\end{equation} 
Indeed, if \eqref{claim1} holds, from the lower semicontinuity of the $\Gamma$-limsup functional $F_+\colon \prob\to  [0,+\infty]$ defined for every $\mu\in\prob$ as $F_+(\mu):=\inf\{\limsup F_n(\mu_n) : \: \mu_n\weak \mu\}$ (see \cite[Proposition 6.8]{dalmaso}) and \eqref{limsup} just proved for piecewise constant measures, it follows that given $\mu\in\prob$ there exists a sequence $\{\mu_n\}$ such that $\mu_n$ is as in Definition~\ref{def.step},  $\mu_n\weak \mu$ and
\[
F_+(\mu) \leq \liminf F_+(\mu_n) \leq \limsup F_\infty(\mu_n) \leq F_\infty(\mu).
\]
This proves the validity of \eqref{limsup} for an arbitrary measure $\mu\in\prob$.

Therefore, we only have to prove the density in energy of the piecewise constant measures in $\prob$. Consider an arbitrary measure $\mu \in\prob$ with density $f\in L^1(I)$ with respect to the Lebesgue measure. Keeping the notation of Definition~\ref{def.step}, we construct the measure $\mu_n\in\prob$ as follows:
\begin{equation*}
\mu_n=f_n(x)\mathcal L, \qquad f_n(x)= \sum_{1\leq i\leq n} \alpha_i^n \chi_{\In}(x), \quad \mathcal L(\In)=1/n
\end{equation*}
where the numbers $\alpha_i^n$ are chosen as to satisfy the conditions
\begin{equation}\label{conda}
\frac{\mu(\In)}{\mathcal L(\In)}\leq  \alpha_i^n \leq\frac{\mu(\overline{\In})}{\mathcal L(\In)},\qquad
\sum_{1\leq i\leq n} \alpha_i^n=n.
\end{equation}
Note that, for fixed $n$, the $\{\In\}$'s are a partition of $I$, and since $\mu_n$ is, by construction, a sort of sampling of $\mu$, it is easy to see that $\mu_n\mathop{\rightharpoonup}^*\mu$ as $n\to\infty$. Moreover, 
\[
\lim_{n\to\infty}\sum_{1\leq i\leq n} \alpha_i^n \varphi\left(\frac{s(x)}{\pi \alpha_i^n} \right)\chi_{I_i^n}(x)=\varphi\left(\frac{s(x)}{\pi f(x)}\right)f(x) \quad\text{a.e. $x\in\{f>0\}$}.
\]
For every $x\in \{f>0\}$ and every sequence of intervals $I_i^n\cap\{f>0\}\downarrow x$ as $n\to\infty$,
by \eqref{conda} we have $\lim_{n\to\infty}\alpha_i^n=f(x)>0$ and moreover $\lim_{n\to\infty} \intmed_{I_i^n\cap \{f>0\}} s(y) dy=s(x)$. Therefore, by the continuity of $\varphi$ there exists a constant $C>1$ such that for every $n$ sufficiently large
\[
\alpha_i^n \varphi\left(\frac{s(x)}{\pi \alpha_i^n} \right)\leq C\bigg(\intmed_{I_i^n\cap\{f>0\}} f\bigg)\varphi\left(\frac{\int_{I_i^n\cap\{f>0\}}(s(y)/(\pi f(y))f(y)dy}{\int_{I_i^n\cap\{f>0\}}f(y)dy} \right)
\]
which by Jensen inequality, recalling \eqref{Fi}, yields
\[
\alpha_i^n \varphi\left(\frac{s(x)}{\pi \alpha_i^n} \right)\leq C\intmed_{I_i^n\cap\{f>0\}}\varphi\left(\frac{s(y)}{\pi f(y)}\right)f(y)dy\leq C F_\infty(\mu).
\]
Now, if $F_\infty(\mu)=+\infty$ the inequality \eqref{limsup} is trivial. Otherwise $F_\infty(\mu)<+\infty$ and by the Lebesgue dominated convergence theorem
\[
\limsup_{n\to\infty}\sum_{1\leq i\leq n} \alpha_i^n\int_{I_i^n} \varphi\left(\frac{s(x)}{\pi \alpha_i^n} \right)dx= \int_I\varphi\left(\frac{s(x)}{\pi f(x)}\right)f(x)dx\leq F_\infty(\mu),
\] 
and \eqref{limsup} holds again. The proposition is then proved.
\end{proof}

\begin{proof}[Proofs of Theorem~\ref{theorem} and of Corollary~\ref{corollary}]
Theorem~\ref{theorem} follows immediately by Proposition~\ref{prop.liminf} and Proposition~\ref{proplimsup}, see \cite[Proposition 8.1]{dalmaso}. Corollary~\ref{corollary} is then a consequence of general results of $\Gamma$-convergence theory. Indeed the space of probability measures $\mathcal P(\overline I)$ is compact w.r.t. the weak* convergence. Moreover, by Jensen inequality it holds that
\[
F_\infty(\mu)\geq \varphi\left(\frac{\frac{1}{\pi}\int_{\{f>0\}}s(x)dx}{\int_I f(x) dx}\right) \int_I f(x) dx+\varphi^\infty \mathcal L(\{f=0\})+ \varphi(0)\mu_s(\overline I),
\] 
and equality holds if and only if $\varphi$ is linear on $\{f>0\}$ or  $\mu$ is absolutely continuous w.r.t. the Lebesgue measure with a density $f=c s(x)\chi_E(x)$ for some measurable set $E\subset I$ and constant $0\leq c\leq 1/\int_Es(x)dx$. By strict convexity of $\varphi$ the latter condition holds.
Therefore, by plugging this function into the right-hand side of the above inequality one gets the following lower bound
\begin{equation}\label{other}
F_\infty(\mu)\geq \varphi\left(\frac{\frac{1}{\pi}\int_Es(x)dx}{c\int_E s(x)dx}\right) c\int_E s(x)dx+\varphi^\infty \mathcal L(I\setminus E)+ \varphi(0)\mu_s(\overline I).
\end{equation}
Then we have to choose the set $E$ and the constant $c$ to minimize this lower bound. We have to face several situations. 

If $\varphi(0)<+\infty$ the fact that $\mu^s(\overline I)=1-\int_I f(x) dx$ allows to write \eqref{other} as
\begin{equation}\label{proofc}
F_\infty(\mu)\geq \left(\varphi\left(\frac{1}{\pi c}\right)-\varphi(0)\right) c\int_Es(x)dx+\varphi^\infty \mathcal L(I\setminus E)+ \varphi(0).
\end{equation}
Moreover, if also $\varphi^\infty=+\infty$ then necessarily $E=I$ (up to a negligible set) and \eqref{proofc} becomes
\begin{equation}\label{proofc2}
F_\infty(\mu)\geq  \left(\varphi\left(\frac{1}{\pi c}\right)-\varphi(0)\right) \pi c \frac{1}{\pi}\int_I s(x)dx + \varphi(0),
\end{equation}
where the right-hand side represents (up to the supplementary constant factors $\varphi(0)$ and ${1}/{\pi}\int_I s(x)dx$)  the slope of the secant line passing through the points with coordinates $(0, \varphi(0))$ and $(1/(\pi c), \varphi(1/(\pi c)))$. By convexity this is clearly minimized when $c$ is as large as possible, namely when $c=1/\int_I s(x)dx$.

If instead $\varphi^\infty=0$ (still under the assumption $\varphi(0)<+\infty$), by convexity $\varphi$ must be non-increasing. Then $(\varphi\left({1}/{(c\pi)}\right)-\varphi(0))\leq 0$ for all admissible constant $c$. By \eqref{proofc} it is again convenient to choose $E$ as large as possible, i.e. $E=I$ (up to a negligible set). Therefore, \eqref{proofc2} holds and as in the previous case we arrive to the same deduction $c=1/ \int_I s(x)dx$.

Now, it remains to consider the case where $\varphi(0)=+\infty$. In this case $\mu^s(\overline I)=0$, then $c\int_E s(x)dx=1$  and  \eqref{other} becomes
\begin{equation*}
F_\infty(\mu)\geq \varphi\left(\frac{1}{\pi}\int_E s(x)dx\right) +\varphi^\infty \mathcal L(I\setminus E).
\end{equation*}
If also $\varphi^\infty=+\infty$ then $E=I$ (up to a negligible set). If instead $\varphi^\infty=0$ then $\varphi$ is non-increasing and as noticed above the minimum is reached for $E=I$ (up to a negligible set).

In conclusion, the $\Gamma$-limit \eqref{Fi} is uniquely minimized by the measure $f_\infty \mathcal L$ with density as in \eqref{minimum}. Therefore, by applying \cite[Corollary 7.24]{dalmaso} we obtain the claim.
\end{proof}

\section{Examples}\label{sec.4}

The general framework introduced before applies to several concrete examples. We illustrate some of them in the following list, focusing on the different representations of $\Gamma$-limits.

\medskip
\noindent1. \emph{Optimal location problems.}
Let $p=w=1$ (then $s=1$) and $q=0$. By \eqref{explicit} problem \eqref{problem} can be written as an optimal partition problem of the lengths of the connected components of the partition, namely
\[
\min \bigg\{ \frac{1}{n}\sum_{1\leq j\leq n}\varphi\left(\frac{n\mathcal L(I_j)}{\pi}\right): \ \{I_j\}\in \mathcal C_n \bigg\}.
\]
This is also related to the so called optimal location problems, see for instance \cite{suzdre, suzoka} where the points have to be located in order to minimize some cost functional depending on the size of the partition. In this setting the $\Gamma$-limit \eqref{Fi} becomes
\[
\Fi(\mu)=\int_I \varphi\left(\frac{1}{\pi f(x)}\right) f(x)\, dx+\varphi^\infty \mathcal L(\{f=0\})+ \varphi(0)\mu^s(\overline I).
\]
and thus, as expected, one find the uniform distribution as minimizer (see \eqref{minimum}).
We point out that in this particular situation it should not be difficult to extend the validity of Theorem~\ref{theorem} also to functions $\varphi$ with linear growth, i.e., the case where $0<\varphi^\infty<+\infty$.

\medskip
\noindent2. \emph{Minimal partitions for positive powers of eigenvalues}.
Let  $r\in (0,+\infty)$ and consider $\varphi(x)=1/x^{2r}$. Then the minimization problem \eqref{problem} reduces to
\[
n^{-2r-1}\min \bigg\{\sum_{1\leq j\leq n} {\la(I_j)^r}: \ \{I_j\}\in \mathcal C_n \bigg\},
\]
where by homogeneity of $\varphi$ the quantity ${n^{-2r-1}}$ can be factorized.
In this case $\varphi^\infty=0$ and $\varphi(0)=+\infty$ thus the $\Gamma$-limit \eqref{Fi} takes the following form
\[
\Fi(\mu)=
\begin{cases}
\displaystyle \pi^{2r} \int_I \frac{f(x)^{2r+1}}{s(x)^{2r}}\, dx,  \quad &\text{if $\mu=f\mathcal L$ with $f\in L^1(I)$,}\\
+\infty \quad & \text{otherwise.}
\end{cases}
\]

\medskip
\noindent3. \emph{Minimal partitions for powers of reciprocal eigenvalues.}
Let  $r\in (1,+\infty)$ and consider $\varphi(x)=x^{2r}$. Then the minimization problem \eqref{problem} becomes
\[
n^{2r-1}\min \bigg\{ \sum_{1\leq j\leq n} \frac{1}{\la(I_j)^r} : \ \{I_j\}\in \mathcal C_n \bigg\},
\]
where as before by homogeneity the quantity ${n^{2r-1}}$ can be factorized. 
Now since $\varphi^\infty=~+\infty$ and $\varphi(0)=0$ the $\Gamma$-limit \eqref{Fi} is
\[
\Fi(\mu)=\frac{1}{\pi^{2r}} \int_I \frac{s(x)^{2r}}{f(x)^{2r-1}}\, dx. 
\]
Notice that this expression encodes the non finiteness of the $\Gamma$-limit when $f$ vanishes somewhere, on a set of positive measure.

\medskip
\noindent4. \emph{Minimal partitions for non-monotone convex functions.}
Let $\varphi(x)=(x-a)^2+b$ for suitable constants $a,b>0$ (such that the assumptions on $\varphi$ stated in the introduction are satisfied). Then the minimization problem \eqref{problem} becomes
\[
\min\bigg\{ \frac{1}{n} \sum_{1\leq j\leq n} \left(\frac{n}{\la(I_j)^{1/2}}-a\right)^2+b : \ \{I_j\}\in \mathcal C_n \bigg\}.
\]
Namely, we would all the eigenvalues $\{\la(I_j)\}$ be as close as possible to $(n/a)^2$, but the fact that the $I_j$'s must be a partition of $I$ imposes a supplementary constraint to the optimization problem. However, still in this case the asymptotics is not affected by the function $\varphi$, which now has a minimum inside $(0,+\infty)$. Moreover, since $\varphi^\infty=+\infty$ and $\varphi(0)=a^2+b$, the $\Gamma$-limit \eqref{Fi} is
\[
\Fi(\mu)=
\displaystyle \frac{1}{\pi^2}\int_I \frac{\left(s(x)-a\pi f(x)\right)^2}{f(x)}\, dx+a^2\mu^s(\overline I)+b, 
\]
and again the fact that $f=0$ somewhere implies the divergence of the integral.

\medskip
\noindent5. \emph{The trace of the heat operator}. If $\varphi(x)=e^{1/x^2}$ then problem \eqref{problem} becomes
\[
\min\bigg\{\frac{1}{n} \sum_{1\leq j\leq n} e^{\la(I_j)/n^2}: \ \{I_j\}\in \mathcal C_n \bigg\}.
\]
Solutions of this problem can be seen as good approximations of the trace of the heat Sturm-Liouville operator near the origin (see \cite{elshar}). Since $\varphi^\infty=0$ and $\varphi(0)=+\infty$, by Theorem~\ref{theorem} one obtains that
\[
\Fi(\mu)=
\begin{cases}
\int_I e^{\left({\pi^2 f(x)^2/s(x)^2}\right)} f(x)\, dx,  \quad &\text{if $\mu=f\mathcal L$ with $f\in L^1(I)$,}\\
+\infty \quad & \text{otherwise.}
\end{cases}
\]

\subsection*{Acknowledgements}
An anonymous referee who carefully read the paper and suggested clarifying remarks is kindly acknowledged.\\
This work is supported by the Research Project INdAM for Young Researchers (Starting Grant) \emph{Optimal Shapes in Boundary Value Problems} and by the INdAM-GNAMPA Project 2018 \emph{Ottimizzazione Geometrica e Spettrale}.

\end{document}